\documentclass[english]{amsart}
\usepackage[T1]{fontenc}
\usepackage[latin9]{inputenc}
\usepackage{amsthm}
\usepackage{amssymb}
\usepackage{esint}

\makeatletter
\numberwithin{equation}{section}
\numberwithin{figure}{section}
\theoremstyle{plain}
\newtheorem{thm}{\protect\theoremname}
  \theoremstyle{plain}
  \newtheorem{prop}[thm]{\protect\propositionname}
  \theoremstyle{plain}
  \newtheorem{cor}[thm]{\protect\corollaryname}
  \theoremstyle{remark}
  \newtheorem{rem}[thm]{\protect\remarkname}
  \theoremstyle{plain}
  \newtheorem{lem}[thm]{\protect\lemmaname}
  \theoremstyle{remark}
  \newtheorem*{acknowledgement*}{\protect\acknowledgementname}

\makeatother

\usepackage{babel}
  \providecommand{\acknowledgementname}{Acknowledgement}
  \providecommand{\corollaryname}{Corollary}
  \providecommand{\lemmaname}{Lemma}
  \providecommand{\propositionname}{Proposition}
  \providecommand{\remarkname}{Remark}
\providecommand{\theoremname}{Theorem}

\begin{document}

\title{On Power Sums of Positive Numbers}

\author{Ruiming Zhang}
\begin{abstract}
In this work we establish a necessary and sufficient condition for
a genus $0$ entire function $f(z)$ has only positive zeros by applying
Hausdorff moment problem and Mergelyan's theorem, the obtained criterion
is very much reminiscent of Xian-Jin Li's criterion on the Riemann
hypothesis. We also apply this criterion to the Riemann hypothesis
and the generalized Riemann hypothesis for certain Dirichlet $L$-series
of real primitive characters. 
\end{abstract}


\curraddr{College of Science\\
Northwest A\&F University\\
Yangling, Shaanxi 712100\\
P. R. China.}


\email{ruimingzhang@yahoo.com}

\maketitle

\section{Introduction }

In this work we are dealing with two problems related to the power
sums for a complex sequence. The first problem is to find a criterion
that a sequence is also positive, and the second problem is to determine
number theoretical properties of a power sum. The first problem is
directly related to the Riemann hypothesis of a zeta function associated
with a special function. We shall apply Hausdorff moment problem and
Mergelyan's theorem to prove a necessary and sufficient condition
for an absolutely summable sequence $\left\{ \lambda_{n}\right\} _{n=1}^{\infty}$
that is also a positive sequence. This criterion is first formulated
in terms of power sums $\sum_{n=1}^{\infty}\lambda_{n}^{k},\ k\in\mathbb{N}$,
then we restate it for an entire function of genus $0$. Given an
even entire function $g(z)$ of at most genus $1$, since the entire
function defined by $f(z)=g(\sqrt{z})$ is of genus $0$, then we
can apply this criterion to obtain necessary and sufficient conditions
for the Riemann hypothesis and generalized Riemann hypothesis of certain
$L$-series. The obtained criterions remind us of Xian-Jin Li's necessary
and sufficient conditions for the Riemann hypothesis\cite{Bombieri,Li}.
Since power sums of a sequence are often special values of a zeta
function, the solution to the second problem gives us an number theoretical
description for these special zeta values. We shall demonstrate that
the power sums of zeros for an entire (basic) hypergeometric function
of genus $0$ are in the field of rational functions in its parameters
over rationals, similar conclusion holds for an even power sum of
zeros for an even entire hypergeometric function of genus $1$.

\section{Preliminaries}

Given an entire function $f(z)={\displaystyle \sum_{n=0}^{\infty}}a_{n}z^{n}$,
its order can be computed by

\[
\rho=\left(1-\limsup_{n\to\infty}\frac{\log\left(|f^{(n)}(z_{0})|\right)}{n\log n}\right)^{-1},
\]
where $z_{0}\in\mathbb{C}$ is an arbitrary point and $f^{(n)}(z)$
is the $n$-th derivative of $f(z)$. The genus for an entire function
of finite order 

\[
f(z)=z^{m}e^{p(z)}\prod_{n=1}^{\infty}\left(1-\frac{z}{z_{n}}\right)\exp\left(\frac{z}{z_{n}}+\ldots+\frac{1}{k}\left(\frac{z}{z_{n}}\right)^{k}\right),
\]
is defined by the nonnegative integer $g=\max\{j,k\}$ where $\left\{ z_{n}\right\} _{n=1}^{\infty}$
all the nonzero roots of $f(z)$, $p(z)$ a polynomial of degree $j$,
and $k$ is the smallest nonnegative integer such that the series
${\displaystyle \sum_{n=1}^{\infty}}\frac{1}{|z_{n}|^{k+1}}$ converges.
If the order $\rho$ is not an integer, then $g$ is the integer part
of $\rho$, $g=\left\lfloor \rho\right\rfloor $. If the order is
a positive integer, then $g$ may be $\rho-1$ or $\rho$. Thus, an
order $0$ entire function also has genus $0$. 

The Riemann Xi function \cite{Csordas,Gasper 1,Gasper-2,Titchmarsh,Zhang}
\[
\Xi(z)=-\frac{(1+4z^{2})}{8\pi^{(1+2iz)/4}}\Gamma\left(\frac{1+2iz}{4}\right)\zeta\left(\frac{1+2iz}{2}\right)
\]
is an even entire function of genus $1$, it satisfies 

\[
\Xi(z)=\int_{-\infty}^{\infty}e^{-itz}\phi(t)dt
\]
and 
\begin{align*}
\Xi(z) & =\sum_{n=0}^{\infty}\frac{(-1)^{n}b_{2n}}{(2n)!}z^{2n},\quad b_{2n}=\int_{-\infty}^{\infty}t^{2n}\phi(t)dt,
\end{align*}
where
\begin{align*}
\phi(t) & =2\pi\sum_{n=1}^{\infty}\left\{ 2\pi n^{4}e^{-\frac{9t}{2}}-3n^{2}e^{-\frac{5t}{2}}\right\} \exp\left(-n^{2}\pi e^{-2t}\right).
\end{align*}
It is well known that $\phi(t)$ is even, positive, smooth and fast
decreasing on $\mathbb{R}$. Evidently, 
\begin{align*}
b_{2n} & >0,\quad\Xi(it)>0,\quad n\in\mathbb{N}_{0}\ t\in\mathbb{R}.
\end{align*}
It is known that all the zeros of $\Xi(z)$ are within the horizontal
strip $|\Im(z)|<\frac{1}{2}$, hence for each real number only finitely
many zeros of $\Xi(z)$ have it as their real part. If we list all
its zeros with positive real part, first according to sizes of their
real parts, then the absolute values of the imaginary parts \cite{Zhang},
\begin{equation}
z_{1},\ z_{2},\dots,\ z_{n},\dots.\label{eq:2.1}
\end{equation}
Then
\begin{equation}
\frac{\Xi(z)}{\Xi(0)}=\prod_{n=1}^{\infty}\left(1-\frac{z^{2}}{z_{n}^{2}}\right),\quad z\in\mathbb{C},\label{eq:2.2}
\end{equation}
which is essentially proved in section 2.8 of \cite{Edwards}. Hence,
\begin{equation}
\prod_{n=1}^{\infty}\left(1-\frac{z}{z_{n}^{2}}\right)=\sum_{n=0}^{\infty}\frac{b_{2n}(-1)^{n}}{(2n)!b_{0}}z^{n},\quad z\in\mathbb{C}\label{eq:2.3}
\end{equation}
defines an entire function such that 
\begin{align}
\prod_{n=1}^{\infty}\left(1-\frac{z}{z_{n}^{2}}\right) & =\frac{\Xi(\sqrt{z})}{\Xi(0)},\quad0\le\arg z<2\pi.\label{eq:2.4}
\end{align}
The Riemann hypothesis is equivalent to the statement that all the
numbers in \eqref{eq:2.1} are positive.

Given $m\in\mathbb{N}$, let $\chi(n)$ be a real primitive character
to modulus $m$ with parity $a$, \cite{Andrews2,Davenport,Titchmarsh,Zhang}
\begin{align*}
a & =\begin{cases}
0, & \chi(-1)=1\\
1, & \chi(-1)=-1
\end{cases}.
\end{align*}
It is known that the even entire function of genus $1$, \cite{Davenport,Edwards,Titchmarsh,Zhang}
\begin{align*}
\Xi(z,\chi) & =\left(\frac{\pi}{m}\right)^{-(1+2a+2iz)/4}\Gamma\left(\frac{1+2a+2iz}{4}\right)L\left(\frac{1+2iz}{2},\chi\right),
\end{align*}
has a Fourier integral representation \cite{Davenport,Edwards,Titchmarsh,Zhang},
\begin{equation}
\Xi(z,\chi)=\int_{-\infty}^{\infty}e^{-izt}\phi(t,\chi)dt,\label{eq:2.5}
\end{equation}
where
\begin{align*}
\phi(t,\chi) & =\sum_{n=-\infty}^{\infty}n^{a}\chi(n)\exp\left(-\frac{n^{2}\pi}{me^{2t}}-\frac{1+a}{2}t\right).
\end{align*}
The function $\phi(t,\chi)$ is clearly smooth and fast decreasing
on $\mathbb{R}$. By applying transformation formulas for a character
$\theta$ function, one can show that $\phi(t,\chi)$ is also even.

From \eqref{eq:2.5} $\Xi(z,\chi)$ has the power series expansion
\begin{align*}
\Xi(z,\chi) & =\sum_{n=0}^{\infty}\frac{(-1)^{n}b_{2n}(\chi)z^{2n}}{(2n)!},\quad b_{2n}(\chi)=\int_{-\infty}^{\infty}t^{2n}\phi(t,\chi)dt.
\end{align*}
In the following discussion we assume the condition 
\begin{equation}
\phi(t,\chi)\ge0,\quad t\in\mathbb{R}\label{eq:2.6}
\end{equation}
holds, so that 
\[
b_{2n}(\chi)>0,\quad\Xi(it,\chi)>0,\quad n\in\mathbb{N}_{0},\ t\in\mathbb{R}.
\]
It is also known that all the zeros of $\Xi(z,\chi)$ are within the
horizontal strip $|\Im(z)|<\frac{1}{2}$, hence for each real number
only finitely many zeros of $\Xi(z,\chi)$ have this number as their
real part. If we list all the zeros with positive real parts, first
according to the sizes of their real parts, then the absolute values
of the imaginary parts, 
\begin{equation}
z_{1}(\chi),z_{2}(\chi),\dots,z_{n}(\chi),\dots.\label{eq:2.7}
\end{equation}
Then from the theory developed in \cite{Davenport,Edwards,Titchmarsh}
we have
\begin{equation}
\frac{\Xi(z,\chi)}{\Xi(0,\chi)}=\prod_{n=1}^{\infty}\left(1-\frac{z^{2}}{z_{n}(\chi)^{2}}\right).\label{eq:2.8}
\end{equation}
Thus, 
\begin{align}
\prod_{n=1}^{\infty}\left(1-\frac{z}{z_{n}(\chi)^{2}}\right) & =\sum_{n=0}^{\infty}\frac{b_{2n}(\chi)(-1)^{n}}{(2n)!b_{0}(\chi)}z^{n}\label{eq:2.9}
\end{align}
defines an entire function such that
\begin{equation}
\prod_{n=1}^{\infty}\left(1-\frac{z}{z_{n}(\chi)^{2}}\right)=\frac{\Xi(\sqrt{z},\chi)}{\Xi(0,\chi)},\quad0\le\arg(z)<2\pi.\label{eq:2.10}
\end{equation}
The generalized Riemann hypothesis for $L(s,\chi)$ is equivalent
to the statement that all numbers in \eqref{eq:2.7} are positive.

From the Euler's infinite product expansion $\frac{\sin\pi z}{\pi z}=\prod_{n=1}^{\infty}\left(1-\frac{z^{2}}{n^{2}}\right)$
we know that 
\begin{equation}
f(z)=\prod_{n=1}^{\infty}\left(1-\frac{z}{n^{2}}\right)=\sum_{n=0}^{\infty}\frac{\pi^{2n}(-z)^{n}}{(2n+1)!}\label{eq:2.11}
\end{equation}
 defines an entire function of genus $0$ with only positive zeros.

For $\nu>-1$, let \cite{Andrews2,Ismail,Watson,Zhang} 
\begin{align*}
0 & <j_{\nu,1}<j_{\nu,2}<\dots<j_{\nu,n},\dots
\end{align*}
be all the positive zeros of the Bessel function
\begin{align*}
J_{\nu}(z) & =\sum_{n=0}^{\infty}\frac{(-1)^{n}}{\Gamma(n+\nu+1)n!}\left(\frac{z}{2}\right)^{\nu+2n}.
\end{align*}
Then $z^{-\nu}J_{\nu}(z)$ is an even entire function of genus $1$
such that 
\[
\Gamma(\nu+1)J_{\nu}(z)\left(\frac{2}{z}\right)^{\nu}=\prod_{n=1}^{\infty}\left(1-\frac{z^{2}}{j_{\nu,n}^{2}}\right).
\]
Then the entire function defined by 
\begin{equation}
\prod_{n=1}^{\infty}\left(1-\frac{z}{j_{\nu,n}^{2}}\right)=\sum_{n=0}^{\infty}\frac{(-1)^{n}z^{n}}{n!2^{2n}(\nu+1)_{n}}\label{eq:2.12}
\end{equation}
satisfies
\begin{equation}
\prod_{n=1}^{\infty}\left(1-\frac{z}{j_{\nu,n}^{2}}\right)=\frac{2^{\nu}\Gamma(\nu+1)J_{\nu}(z^{1/2})}{z^{\nu/2}},\quad0\le\arg(z)<2\pi,\label{eq:2.13}
\end{equation}
where 
\[
(a)_{n}=\frac{\Gamma(a+n)}{\Gamma(a)},\quad a,n\in\mathbb{C},\ a+n\not\in-\mathbb{N}_{0}.
\]
The modified Bessel function \cite{Andrews2,Gasper-2,Ismail,Watson}
\begin{equation}
K_{iz}(a)=\int_{0}^{\infty}e^{-a\cosh u}\cos(zu)du,\quad a>0\label{eq:2.14}
\end{equation}
is an entire function in variable $z$. It is known that for $a>0$,
$K_{iz}(a)$ is an even entire function of genus $1$ in variable
$z$ that has only real zeros. 

Let 
\begin{align*}
0 & <i_{1}<i_{2}<\dots
\end{align*}
be the zeros of the Airy function \cite{Andrews2,Ismail,Zhang} 
\begin{align*}
A(z) & =\frac{\pi}{\sqrt[3]{3}}\mathrm{Ai}\left(-\frac{z}{\sqrt[3]{3}}\right)=\frac{\pi}{3}\sqrt{\frac{x}{3}}\left\{ J_{-\frac{1}{3}}\left(2\left(\frac{x}{3}\right)^{\frac{3}{2}}\right)+J_{\frac{1}{3}}\left(2\left(\frac{x}{3}\right)^{\frac{3}{2}}\right)\right\} .
\end{align*}
Then it is an entire function of genus $1$ that has only positive
zeros, and it satisfies \cite{Zhang}
\begin{equation}
\prod_{n=1}^{\infty}\left(1-\frac{z^{2}}{i_{n}^{2}}\right)=\frac{9\Gamma(2/3)^{2}A(z)A(-z)}{\pi^{2}}.\label{eq:2.15}
\end{equation}
Then the entire function defined by 
\begin{equation}
\prod_{n=1}^{\infty}\left(1-\frac{z}{i_{n}^{2}}\right)=\sum_{n=0}^{\infty}(-1)^{n}a_{n}z^{n}\label{eq:2.16}
\end{equation}
satisfies 
\begin{equation}
\prod_{n=1}^{\infty}\left(1-\frac{z}{i_{n}^{2}}\right)=\frac{9\Gamma(2/3)^{2}A(\sqrt{z})A(-\sqrt{z})}{\pi^{2}},\quad0\le\arg(z)<2\pi,\label{eq:2.17}
\end{equation}
where
\begin{align}
a_{n} & =\frac{\sqrt{3}\Gamma\left(\frac{2}{3}\right)^{2}}{\sqrt[3]{4}\pi}\frac{16^{\frac{n}{3}}\Gamma\left(\frac{n}{3}+\frac{1}{6}\right)\Gamma\left(\frac{n}{3}+\frac{1}{2}\right)}{(2n)!}.\label{eq:2.18}
\end{align}
Assume that $0<q<1$, let \cite{Andrews2,Ismail,Zhang}
\begin{align*}
(z;q)_{\infty} & =\prod_{n=0}^{\infty}(1-zq^{n}),\ (z;q)_{n}=\frac{(z;q)_{\infty}}{(zq^{n};q)_{\infty}}
\end{align*}
and
\begin{align*}
(z_{1},z_{2},\dots,z_{m};q)_{n} & =\prod_{j=1}^{m}(z_{j};q)_{n}
\end{align*}
for all $m\in\mathbb{N}$,$n\in\mathbb{Z}$ and $z,z_{1},z_{2},\dots,z_{m}\in\mathbb{C}$. 

For $\nu>-1$, let 
\begin{align*}
0 & <j_{\nu,1}(q)<j_{\nu,2}(q)<\dots<j_{\nu,n}(q)<\dots
\end{align*}
be all the positive zeros of the $q$-Bessel function \cite{Andrews2,Ismail}

\begin{align*}
J_{\nu}^{(2)}(z;q) & =\frac{(q^{\nu+1};q)_{\infty}}{(q;q)_{\infty}}\sum_{n=0}^{\infty}\frac{\left(-q^{n+\nu}\right)^{n}}{(q,q^{\nu+1};q)_{n}}\left(\frac{z}{2}\right)^{\nu+2n}.
\end{align*}
Then the genus $0$ even entire function $\left(\frac{2}{z}\right)^{\nu}J_{\nu}^{(2)}(z;q)$
satisfies
\[
\left(\frac{2}{z}\right)^{\nu}J_{\nu}^{(2)}(z;q)=\frac{(q^{\nu+1};q)_{\infty}}{(q;q)_{\infty}}\prod_{n=1}^{\infty}\left(1-\frac{z^{2}}{j_{\nu,n}^{2}(q)}\right).
\]
Thus the entire function defined by 
\begin{align}
\prod_{n=1}^{\infty}\left(1-\frac{z}{j_{\nu,n}^{2}(q)}\right) & =\sum_{n=0}^{\infty}\frac{(-1)^{n}q^{n(n+\nu)}z^{n}}{(q,q^{\nu+1};q)_{n}2^{2n}}\label{eq:2.19}
\end{align}
satisfies 
\begin{equation}
\prod_{n=1}^{\infty}\left(1-\frac{z}{j_{\nu,n}^{2}(q)}\right)=\frac{2^{\nu}(q;q)_{\infty}J_{\nu}^{(2)}(z^{1/2};q)}{(q^{\nu+1};q)_{\infty}z^{\nu/2}},\quad0\le\arg z<2\pi.\label{eq:2.20}
\end{equation}
Let 
\begin{align*}
0 & <i_{1}(q)<i_{2}(q)<\dots<i_{n}(q)<\dots
\end{align*}
be all the zeros of the Ramanujan's entire function\cite{Andrews1,Andrews2,Ismail},
then 
\begin{align}
A_{q}(z) & =\sum_{n=0}^{\infty}\frac{q^{n^{2}}(-z)^{n}}{(q;q)_{n}}=\prod_{n=1}^{\infty}\left(1-\frac{z}{i_{n}(q)}\right)\label{eq:2.21}
\end{align}
is an entire function of genus $0$ with only positive zeros.

\section{Main Results}

A convenient proposition to determine the number theoretical properties
of certain power sums:
\begin{prop}
\label{prop:3.1} Let ${\displaystyle \sum_{n=1}^{\infty}|}\lambda_{n}|<\infty$,
$p_{k}={\displaystyle \sum_{n=1}^{\infty}}\lambda_{n}^{k}$, and 
\begin{equation}
f(z)=\prod_{n=1}^{\infty}\left(1-\lambda_{n}z\right)=\sum_{k=0}^{\infty}(-1)^{k}e_{k}z^{k},\quad z\in\mathbb{C}.\label{eq:3.1}
\end{equation}
If $\left\{ e_{k}\right\} _{n=0}^{\infty}\subset\mathbb{Q}$, then
$\left\{ p_{k}\right\} _{n=1}^{\infty}\subset\mathbb{Q}$. Let $\ell\in\mathbb{N}$
and $\mathbb{Q}(q_{1},\dots,q_{\ell})$ be the field of rational functions
in variables $q_{1},q_{2},\dots q_{\ell}$. If $\left\{ e_{k}\right\} _{n=0}^{\infty}\subset\mathbb{Q}(q_{1},\dots,q_{\ell})$,
then $\left\{ p_{k}\right\} _{n=1}^{\infty}\subset\mathbb{Q}(q_{1},\dots,q_{\ell})$.
In particular, let $f(z)$ be a (basic) hypergeometric type entire
function of genus $0$, then $\left\{ p_{k}\right\} _{n=1}^{\infty}\subset\mathbb{Q}((q),\ a_{1},\dots,a_{r},\ b_{1},\dots b_{s})$,
where $(q),\ a_{1},\dots,a_{r},\ b_{1},\dots b_{s}$ are the parameters,
and $(q)$ means that it appears only when $f(z)$ is a $q$-series.
Here are some examples:
\begin{enumerate}
\item for $n\in\mathbb{N}$, $\zeta(2n)\in\mathbb{Q}(\pi^{2})$.
\item for $\nu>-1,\ n\in\mathbb{N}$, ${\displaystyle \sum_{k=1}^{\infty}}\frac{1}{j_{\nu,k}^{2n}}\in\mathbb{Q}(\nu)$.
\item for $n\in\mathbb{N}$, ${\displaystyle \sum_{k=1}^{\infty}}\frac{1}{i_{k}^{2n}}\in\mathbb{Q}\left(\sqrt{3},\pi,\Gamma\left(\frac{1}{3}\right)\right)$.
\item for $\nu>-1,\,0<q<1,\,n\in\mathbb{N}$, ${\displaystyle \sum_{k=1}^{\infty}}\frac{1}{j_{\nu,k}^{2n}(q)}\in\mathbb{Q}(q,\ q^{\nu})$. 
\item for $0<q<1,\,n\in\mathbb{N}$,${\displaystyle \sum_{k=1}^{\infty}}\frac{1}{i_{k}^{n}(q)}\in\mathbb{Q}(q)$. 
\item for $n\in\mathbb{N}$, ${\displaystyle \sum_{k=1}^{\infty}\frac{1}{z_{k}^{2n}}\in\mathbb{Q}\left(\frac{b_{2}}{b_{0}},\dots,\frac{b_{2n}}{b_{0}}\right)}$. 
\item for $n\in\mathbb{N}$, ${\displaystyle \sum_{k=1}^{\infty}\frac{1}{z_{k}(\chi)^{2n}}\in\mathbb{Q}\left(\frac{b_{2}(\chi)}{b_{0}(\chi)},\dots,\frac{b_{2n}(\chi)}{b_{0}(\chi)}\right)}$.
\end{enumerate}
\end{prop}
A positivity criterion for an absolutely summable sequence in terms
of power sums:
\begin{thm}
\label{thm:3.2} Assume that $\left\{ \lambda_{n}\right\} _{n=1}^{\infty}\in\mathbb{C}$
is an absolutely summable sequence such that not all of them are zeros.
Let
\begin{equation}
p_{k}={\displaystyle \sum_{n=1}^{\infty}}\lambda_{n}^{k},\ k\in\mathbb{N}.\label{eq:3.2}
\end{equation}
Then, $\left\{ \lambda_{n}\right\} _{n=1}^{\infty}$ is a sequence
of positive numbers if and only if for some $\lambda\ge\sup\left\{ \left|\lambda_{n}\right|\bigg|\ n\in\mathbb{N}\right\} $,
\begin{equation}
\ensuremath{(-1)^{j}\Delta^{j}\left(\frac{p_{k+1}}{\lambda^{k+1}}\right)\ge0},\quad j,k\in\mathbb{N}_{0},\label{eq:3.3}
\end{equation}
where $\Delta m_{n}=m_{n+1}-m_{n}$.

\end{thm}
Taking $\lambda_{k}=\frac{1}{z_{k}(\chi,a)^{2}}$ in Theorem \ref{thm:3.2}
to get the following:
\begin{cor}
\label{cor:3.3} Given $m\in\mathbb{N}$, let $\chi(n)$ be an non-principal
real primitive character to modulus $m$ with parity $a$ such that
$\phi(t,\chi)\ge0,\quad t\in\mathbb{R}$. Then the generalized Riemann
hypothesis is true for $L(s,\chi)$ if and only if for some $\lambda\ge\sup\left\{ \frac{1}{\left|z_{k}(\chi)\right|^{2}}\bigg|\ k\in\mathbb{N}\right\} $,
\begin{equation}
\ensuremath{(-1)^{j}\Delta^{j}\left(\frac{p_{k+1}(\chi)}{\lambda^{k+1}}\right)\ge0},\quad j,k\in\mathbb{N}_{0},\label{eq:3.4}
\end{equation}
where 
\begin{equation}
p_{n}(\chi)=\sum_{k=1}^{\infty}\frac{1}{z_{k}(\chi)^{2n}}=\frac{-1}{(n-1)!}\frac{\partial^{n}}{\partial z^{n}}\left(\log\frac{\Xi(\sqrt{z},\chi)}{\Xi(0,\chi)}\right)\bigg|_{z=0},\quad n\in\mathbb{N},\label{eq:3.5}
\end{equation}
and
\begin{align}
p_{1}(\chi) & =\frac{b_{2}(\chi)}{2b_{0}(\chi)},\quad p_{2}(\chi)=\frac{3b_{2}^{2}(\chi)-b_{0}(\chi)b_{4}(\chi)}{12b_{0}^{2}(\chi)},\label{eq:3.6}\\
p_{3}(\chi) & =\frac{30b_{2}^{3}(\chi)-15b_{0}(\chi)b_{2}(\chi)b_{4}(\chi)+b_{0}^{2}(\chi)b_{6}(\chi)}{240b_{0}^{3}(\chi)},\label{eq:3.7}
\end{align}
\begin{equation}
p_{k}(\chi)=\frac{(-1)^{k-1}kb_{2k}(\chi)}{(2k)!b_{0}(\chi)}+\sum_{i=1}^{k-1}\frac{(-1)^{k-1+i}b_{2k-2i}(\chi)}{(2k-2i)!b_{0}(\chi)}p_{i}(\chi),\label{eq:3.8}
\end{equation}
 
\begin{equation}
p_{k}(\chi)=\sum_{\begin{array}{c}
r_{1}+r_{2}+\cdots+jr_{j}=k\\
r_{1}\ge0,\dots,r_{j}\ge0
\end{array}}(-1)^{k}\frac{k\left(r_{1}+\cdots+r_{j}-1\right)!}{r_{1}!r_{2}!\cdots r_{j}!}\prod_{i=1}^{j}\left(\frac{-b_{2i}(\chi)}{(2i)!b_{0}(\chi)}\right)^{r_{i}}.\label{eq:3.9}
\end{equation}

\end{cor}
Taking $\lambda_{k}=\frac{1}{z_{k}^{2}}$ in Theorem \ref{thm:3.2}
to get:
\begin{cor}
\label{cor:3.4} The Riemann hypothesis is valid if and only if for
some fixed $\lambda\ge\sup\left\{ \frac{1}{\left|z_{k}\right|^{2}}\bigg|\ k\in\mathbb{N}\right\} $,
\begin{equation}
\ensuremath{(-1)^{j}\Delta^{j}\left(\frac{p_{k+1}}{\lambda^{k+1}}\right)\ge0},\quad j,k\in\mathbb{N}_{0},\label{eq:3.10}
\end{equation}
where 
\begin{equation}
p_{n}=\sum_{k=1}^{\infty}\frac{1}{z_{k}^{2n}}=\frac{-1}{(n-1)!}\frac{\partial^{n}}{\partial z^{n}}\left(\log\frac{\Xi(\sqrt{z})}{\Xi(0)}\right)\bigg|_{z=0},\quad n\in\mathbb{N}\label{eq:3.11}
\end{equation}

\end{cor}
\begin{equation}
p_{k}=\frac{(-1)^{k-1}kb_{2k}}{(2k)!b_{0}}+\sum_{i=1}^{k-1}\frac{(-1)^{k-1+i}b_{2k-2i}p_{i}}{(2k-2i)!b_{0}},\label{eq:3.12}
\end{equation}
\begin{equation}
p_{k}=\sum_{\begin{array}{c}
r_{1}+r_{2}+\cdots+jr_{j}=k\\
r_{1}\ge0,\dots,r_{j}\ge0
\end{array}}(-1)^{k}\frac{k\left(r_{1}+\cdots+r_{j}-1\right)!}{r_{1}!r_{2}!\cdots r_{j}!}\prod_{i=1}^{j}\left(-\frac{b_{2i}}{(2i)!b_{0}}\right)^{r_{i}}.\label{eq:3.13}
\end{equation}
The first 4 $p_{n}$s are:

\begin{align}
p_{1} & =\frac{b_{2}}{2b_{0}},\quad p_{2}=\frac{3b_{2}^{2}-b_{0}b_{4}}{12b_{0}^{2}},\quad p_{3}=\frac{30b_{2}^{3}-15b_{0}b_{2}b_{4}+b_{0}^{2}b_{6}}{240b_{0}^{3}},\label{eq:3.14}\\
p_{4} & =\frac{630b_{2}^{4}-420b_{0}b_{2}^{2}b_{4}+35b_{0}^{2}b_{4}^{2}+28b_{0}^{2}b_{2}b_{6}-b_{0}^{3}b_{8}}{10080b_{0}^{4}}.\label{eq:3.15}
\end{align}
 A positivity criterion for zeros of an entire function of genus $0$:
\begin{thm}
\label{thm:3.5} Assume that $f(z)$ is an entire function of genus
$0$ such that $f(0)=1$. Then $f(z)$ has only positive roots if
and only if for some positive number $\rho$ such that $0<\rho\le\inf\left\{ \left|z\right|\bigg|\ f(z)=0\right\} $,
and for all nonnegative integers $j,\:k$ we have
\begin{equation}
\frac{\partial^{j+k}}{\partial z^{j+k}}\left\{ \left(z-1\right)^{j}\frac{f'(\rho z)}{f(\rho z)}\right\} \bigg|_{z=0}\le0.\label{eq:3.16}
\end{equation}

\end{thm}
From Corollary \ref{cor:3.3} we get:
\begin{cor}
\label{cor:3.6} Given $m\in\mathbb{N}$, let $\chi(n)$ be an non-principal
real primitive character to modulus $m$ with parity $a$ such that
$\phi(t,\chi)\ge0,\quad t\in\mathbb{R}$. Then the generalized Riemann
hypothesis is true for $L(s,\chi)$ if and only if for some positive
number $\rho$ such that $0<\rho\le\inf\left\{ \left|z\right|\bigg|\ \Xi(z,\chi)=0\right\} $,
and for all nonnegative integers $j,\:k$ we have
\begin{equation}
\frac{\partial^{j+k}}{\partial z^{j+k}}\left\{ \left(z-1\right)^{j}\frac{\Xi'(\rho\sqrt{z},\chi)}{\sqrt{z}\Xi(\rho\sqrt{z},\chi)}\right\} \bigg|_{z=0}\le0.\label{eq:3.17}
\end{equation}

\end{cor}
From Corollary \ref{cor:3.4} we get:
\begin{cor}
\label{cor:3.7}The Riemann hypothesis holds if and only if for some
positive number $\rho$ such that $0<\rho\le\inf\left\{ \left|z\right|\bigg|\ \Xi(z)=0\right\} $,
and for all nonnegative integers $j,\:k$ we have
\begin{equation}
\frac{\partial^{j+k}}{\partial z^{j+k}}\left\{ \left(z-1\right)^{j}\frac{\Xi'(\rho\sqrt{z})}{\sqrt{z}\Xi(\rho\sqrt{z})}\right\} \bigg|_{z=0}\le0.\label{eq:3.18}
\end{equation}

\end{cor}
Applying Theorem \ref{thm:3.5} to even entire functions of genus
at most $1$:
\begin{cor}
\label{cor:3.8} Let $G(z)$ be an even entire function of genus at
most $1$ such that it takes real values for real $z$, and $G(0)\neq0$.
If $G(z)$ has only real roots, then for all $c\in\mathbb{R}$ and
for all nonnegative integers $j,\,k$ we have

\begin{equation}
\frac{\partial^{j+k}}{\partial z^{j+k}}\left\{ \left(z-1\right)^{j}\frac{G'(\rho\sqrt{z}-ic)+G'(\rho\sqrt{z}+ic)}{\sqrt{z}\left(G(\rho\sqrt{z}-ic)+G(\rho\sqrt{z}+ic)\right)}\right\} \bigg|_{z=0}\le0,\label{eq:3.19}
\end{equation}
where $\rho$ is a fixed positive number such that $0<\rho<\inf\left\{ \left|z\right|\bigg|\ G(z)=0\right\} $. \end{cor}
\begin{rem}
Inequalities \eqref{eq:3.19} are valid for the following special
functions:
\begin{enumerate}
\item $G(z)=\frac{\sin\pi z}{\pi z}$; 
\item $G(z)=K_{iz}(a),\quad a>0$;
\item $G(z)=\frac{J_{\nu}(z)}{z^{\nu}},\quad\nu>-1$;
\item $G(z)=\frac{J_{\nu}^{(2)}(z;q)}{z^{\nu}},\quad q\in(0,1),\quad\nu>-1$;
\item $G(z)=\Xi(z)$ under the assumption of the Riemann hypothesis;
\item $G(z)=\Xi(z,\chi)$ under the assumptions of the generalized Riemann
hypothesis and \eqref{eq:2.6};
\item $G(z)=f(z)f(-z)$, where $f(z)$ is an entire function of genus at
most $1$, takes real values for real $z$, and $f(0)\neq0$. In particular,
for $G(z)=A(z)A(-z)$ and $G(z)=A_{q}(z)A_{q}(-z)$.
\end{enumerate}
\end{rem}

\section{Proofs }

The following lemma gives a recurrence and a closed formula to express
power sums $p_{k}$ in terms of multi-sums $e_{k}$.
\begin{lem}
\label{lem:4.1} Let $\left\{ \lambda_{n}\right\} _{n=1}^{\infty}$
be a sequence of complex numbers such that ${\displaystyle \sum_{n=1}^{\infty}}|\lambda_{n}|<\infty$.
Then both series 
\begin{equation}
p_{k}=\sum_{i=1}^{\infty}\lambda_{i}^{k},\quad e_{k}=\sum_{1\le j_{1}<\dots<j_{k}}\lambda_{j_{1}}\dots\lambda_{j_{k}}\label{eq:4.1}
\end{equation}
 converge absolutely, and they satisfy
\begin{equation}
p_{k}=(-1)^{k-1}ke_{k}+\sum_{i=1}^{k-1}(-1)^{k-1+i}e_{k-i}p_{i}\label{eq:4.2}
\end{equation}
 and
\begin{equation}
p_{k}=\sum_{\begin{array}{c}
r_{1}+r_{2}+\cdots+jr_{j}=k\\
r_{1}\ge0,\dots,r_{j}\ge0
\end{array}}(-1)^{k}\frac{k\left(r_{1}+\cdots+r_{j}-1\right)!}{r_{1}!r_{2}!\cdots r_{j}!}\prod_{i=1}^{j}\left(-e_{i}\right)^{r_{i}}.\label{eq:4.3}
\end{equation}
\end{lem}
\begin{proof}
Let $x_{1},x_{2},\dots,x_{n}\in\mathbb{C}$, \cite{Mead}, 
\[
p_{k}^{(n)}=p_{k}(x_{1},\dots,x_{n})=\sum_{i=1}^{n}x_{i}^{k},\quad k\in\mathbb{N},
\]
and 
\[
e_{k}^{(n)}=e_{k}(x_{1},\dots,x_{n})=\sum_{1\le j_{1}<\dots<j_{k}\le n}x_{j_{1}}\dots x_{j_{k}},\quad k\in\mathbb{N}_{0}
\]
with conventions 
\[
e_{k}^{(n)}=e_{k}(x_{1},\dots,x_{n})=0,\quad k>n.
\]
Then for $k,\,n\ge1$, we have the following recurrence \cite{Mead}
\begin{equation}
p_{k}^{(n)}=(-1)^{k-1}ke_{k}^{(n)}+\sum_{i=1}^{k-1}(-1)^{k-1+i}e_{k-i}^{(n)}p_{i}^{(n)}\label{eq:4.4}
\end{equation}
and the closed formula,
\begin{equation}
p_{k}^{(n)}=\sum_{\begin{array}{c}
r_{1}+r_{2}+\cdots+jr_{j}=k\\
r_{1}\ge0,\dots,r_{j}\ge0
\end{array}}(-1)^{k}\frac{k\left(r_{1}+\cdots+r_{j}-1\right)!}{r_{1}!r_{2}!\cdots r_{j}!}\prod_{i=1}^{j}\left(-e_{i}^{(n)}\right)^{r_{i}}.\label{eq:4.5}
\end{equation}
Since 
\[
\sum_{n=1}^{\infty}|\lambda_{n}|<\infty,\quad\sup_{i\ge1}\left|\lambda_{i}\right|<\infty,
\]
then, 
\[
\sum_{i=1}^{\infty}\left|\lambda_{i}^{k}\right|\le\left(\sup_{i\ge1}\left|\lambda_{i}\right|\right)^{k-1}\sum_{i=1}^{\infty}\left|\lambda_{i}\right|<\infty
\]
 and 
\[
\sum_{1\le j_{1}<\dots<j_{k}}\left|\lambda_{j_{1}}\dots\lambda_{j_{k}}\right|\le\left(\sum_{i=1}^{\infty}\left|\lambda_{i}\right|\right)^{k}<\infty.
\]
 For any $N\in\mathbb{N}$, let 
\[
p_{k}^{(N)}=\sum_{i=1}^{N}\lambda_{i}^{k},\quad e_{k}^{(N)}=\sum_{1\le j_{1}<\dots<j_{k}\le N}\lambda_{j_{1}}\dots\lambda_{j_{k}}.
\]
 Then we have
\[
\left|p_{k}-p_{k}^{(N)}\right|\le\sum_{i=N+1}^{\infty}\left|\lambda_{i}^{k}\right|\le\left(\sup_{i\ge N+1}\left|\lambda_{i}\right|\right)^{k-1}\sum_{i=N+1}^{\infty}\left|\lambda_{i}\right|
\]
 and
\[
\left|e_{k}-e_{k}^{(N)}\right|\le\left(\sum_{i=1}^{\infty}\left|\lambda_{i}\right|\right)^{k-1}\sum_{i=N+1}^{\infty}\left|\lambda_{i}\right|.
\]
Since
\[
\lim_{N\to\infty}\sum_{i=N+1}^{\infty}\left|\lambda_{i}\right|=\lim_{N\to\infty}\sup_{i\ge N+1}\left|\lambda_{i}\right|=0,
\]
then,
\[
\lim_{N\to\infty}p_{k}^{(N)}=p_{k},\quad\lim_{N\to\infty}e_{k}^{(N)}=e_{k}.
\]
From \eqref{eq:4.4} and \eqref{eq:4.5} we get 
\begin{equation}
p_{k}^{(N)}=(-1)^{k-1}ke_{k}^{(N)}+\sum_{i=1}^{k-1}(-1)^{k-1+i}e_{k-i}^{(N)}p_{i}^{(N)}\label{eq:4.6}
\end{equation}
and 
\begin{equation}
p_{k}^{(N)}=\sum_{\begin{array}{c}
r_{1}+r_{2}+\cdots+jr_{j}=k\\
r_{1}\ge0,\dots,r_{j}\ge0
\end{array}}(-1)^{k}\frac{k\left(r_{1}+\cdots+r_{j}-1\right)!}{r_{1}!r_{2}!\cdots r_{j}!}\prod_{i=1}^{j}\left(-e_{i}^{(N)}\right)^{r_{i}}.\label{eq:4.7}
\end{equation}
We observe that since $k$ here is a fixed positive integer, then
both of \eqref{eq:4.6} and \eqref{eq:4.7} are relations of finite
terms. Thus we can take limit $N\to\infty$ in these identities to
get \eqref{eq:4.2} and \eqref{eq:4.3} respectively. 
\end{proof}
The following lemma gives a method to express the power sums $p_{k}$
and multi-sums $e_{k}$ in terms of the logarithmic derivatives and
derivatives of their associated entire function $f(z)$ at $z=0$.
\begin{lem}
\label{lem:4.5} Let $\left\{ \lambda_{n}\right\} _{n=1}^{\infty}$
be an absolutely summable sequence of complex numbers and let 
\begin{equation}
f(z)=\prod_{n=1}^{\infty}\left(1-\lambda_{n}z\right).\label{eq:4.8}
\end{equation}
Then we have 
\begin{equation}
e_{k}=\frac{(-1)^{k}}{k!}\frac{\partial^{k}f(z)}{\partial z^{k}}\bigg|_{z=0},\quad k\in\mathbb{N}_{0}\label{eq:4.9}
\end{equation}
and

\begin{equation}
p_{k}=\frac{-1}{(k-1)!}\frac{\partial^{k}}{\partial z^{k}}\log f(z)\bigg|_{z=0},\quad k\in\mathbb{N}.\label{eq:4.10}
\end{equation}
\end{lem}
\begin{proof}
For any $N\in\mathbb{N}$, clearly,
\begin{equation}
\prod_{n=1}^{N}\left(1-\lambda_{n}z\right)=\sum_{k=0}^{N}(-1)^{k}e_{k}^{(N)}z^{k}.\label{eq:4.11}
\end{equation}
Since $\sum_{n=1}^{\infty}|\lambda_{n}|<\infty$, then for $z$ in
any compact subset of $\mathbb{C}$ we have
\[
\lim_{N\to\infty}\prod_{n=1}^{N}\left(1-\lambda_{n}z\right)=\prod_{n=1}^{\infty}\left(1-\lambda_{n}z\right)
\]
 uniformly. On the other hand, we observe that
\[
|e_{k}^{(N)}|\le\sum_{1\le j_{1}<\dots<j_{k}\le N}\left|\lambda_{j_{1}}\right|\cdot\dots\cdot\left|\lambda_{j_{k}}\right|\le\left(\sum_{i=1}^{\infty}\left|\lambda_{i}\right|\right)^{k},
\]
 hence,
\[
\left|e_{k}^{(N)}z^{k}\right|\le\left(|z|\cdot\sum_{i=1}^{\infty}\left|\lambda_{i}\right|\right)^{k}.
\]
For $|z|\le\left(\sum_{n=1}^{\infty}|\lambda_{n}|+1\right)^{-2}$,
the right hand side series of \eqref{eq:4.11} converges absolutely
and uniformly, then,
\[
f(z)=\lim_{N\to\infty}\sum_{k=0}^{N}(-1)^{k}e_{k}^{(N)}z^{k}=\sum_{k=0}^{\infty}(-1)^{k}e_{k}z^{k},
\]
and \eqref{eq:4.9} is proved. For $|z|\le\left(\sum_{n=1}^{\infty}|\lambda_{n}|+1\right)^{-2}$,
the entire function $f(z)=\prod_{n=1}^{\infty}\left(1-\lambda_{n}z\right)\ne0,$
and it converges uniformly and absolutely, then,
\[
\frac{\partial^{k}}{\partial z^{k}}\log f(z)=-(k-1)!\sum_{n=1}^{\infty}\frac{\lambda_{n}^{k}}{(1-\lambda_{n}z)^{k}}
\]
and \eqref{eq:4.10}follows. \end{proof}
\begin{rem}
Notice in the above proof, we only need the fact both series are convergent
in a neighborhood of $0$. But it is possible to show both series
converge on the punctured complex plane by considering their tails
instead.\end{rem}
\begin{lem}
\label{lem:4.3} Assume that the generating function $f(z)={\displaystyle \sum_{n=0}^{\infty}}a_{n}z^{n}$
for a complex sequence $\left\{ a_{n}\right\} _{n=0}^{\infty}$ is
analytic at $z=0$. Then for all nonnegative integers $j,\ n$ we
have 
\begin{equation}
\left(-\Delta\right)^{j}a_{n}=\frac{1}{\left(n+j\right)!}\frac{\partial^{j+n}}{\partial z^{j+n}}\left\{ \left(z-1\right)^{j}f(z)\right\} \bigg|_{z=0}.\label{eq:4.12}
\end{equation}
\end{lem}
\begin{proof}
First we observe that 
\begin{equation}
\left(-\Delta\right)^{j}a_{n}=\sum_{k=0}^{j}\binom{j}{k}(-1)^{k}a_{n+k}.\label{eq:4.13}
\end{equation}
Clearly, it is true for $j=0$. Assume that \eqref{eq:4.13} holds
for any $j\ge0$, then,
\begin{eqnarray*}
\left(-\Delta\right)^{j+1}a_{n} & = & \sum_{k=0}^{j}\binom{j}{k}(-1)^{k}a_{n+k}-\sum_{k=0}^{j}\binom{j}{k}(-1)^{k}a_{n+k+1}\\
 & = & \sum_{k=0}^{j+1}\binom{j}{k}(-1)^{k}a_{n+k}+\sum_{k=1}^{j+1}\binom{j}{k-1}(-1)^{k}a_{n+k}\\
 & = & \sum_{k=0}^{j+1}\left\{ \binom{j}{k}+\binom{j}{k-1}\right\} (-1)^{k}a_{n+k}=\sum_{k=0}^{j+1}\binom{j+1}{k}(-1)^{k}a_{n+k}
\end{eqnarray*}
 and \eqref{eq:4.13} is proved by mathematical induction. 

Observe that 
\begin{eqnarray*}
\left(1-z\right)^{j}f(z) & = & \sum_{k=0}^{\infty}\binom{j}{k}(-1)^{k}z^{k}\sum_{\ell=0}^{\infty}a_{\ell}z^{\ell}=\sum_{n=0}^{\infty}z^{n}\sum_{k=0}^{\min\left\{ j,n\right\} }\binom{j}{k}(-1)^{k}a_{n-k}\\
 & = & \sum_{n=0}^{j-1}z^{n}\sum_{k=0}^{n}\binom{j}{k}(-1)^{k}a_{n-k}+\sum_{n=j}^{\infty}z^{n}\sum_{k=0}^{j}\binom{j}{j-k}(-1)^{k}a_{n-k}\\
 & = & \sum_{n=0}^{j-1}z^{n}\sum_{k=0}^{n}\binom{j}{k}(-1)^{k}a_{n-k}+(-1)^{j}\sum_{n=j}^{\infty}z^{n}\sum_{k=0}^{j}\binom{j}{k}(-1)^{k}a_{n-j+k}\\
 & = & \sum_{n=0}^{j-1}z^{n}\sum_{k=0}^{n}\binom{j}{k}(-1)^{k}a_{n-k}+(-z)^{j}\sum_{m=0}^{\infty}z^{m}\sum_{k=0}^{j}\binom{j}{k}(-1)^{k}a_{m+k}\\
 & = & \sum_{n=0}^{j-1}z^{n}\sum_{k=0}^{n}\binom{j}{k}(-1)^{k}a_{n-k}+(-z)^{j}\sum_{m=0}^{\infty}z^{m}\left(-\Delta\right)^{j}a_{m}.
\end{eqnarray*}
 Hence for $j,\,m\ge0$ we have
\[
\left(-\Delta\right)^{j}a_{m}=\frac{\left(-1\right)^{j}}{\left(m+j\right)!}\frac{\partial^{j+m}}{\partial z^{j+m}}\left\{ \left(1-z\right)^{j}f(z)\right\} \bigg|_{z=0}.
\]

\end{proof}

\subsection{Proof of Proposition \ref{prop:3.1}}

The main assertion of this proposition follows from formulas \eqref{eq:4.3}
and \eqref{eq:4.9}, the same conclusion can be drawn by applying
mathematical induction to the recurrence \eqref{eq:4.2}. 

Example 1 follows from applying the main assertion to \eqref{eq:2.11}
and $\lambda_{k}=\frac{1}{k^{2}}$, example 2 follows from \eqref{eq:2.12}
and $\lambda_{k}=\frac{1}{j_{\nu,k}^{2}}$. For example 3, we first
observe that $\mathbb{Q}\left(a_{0},a_{1},a_{2},\dots,\,a_{n}\right)=\mathbb{Q}\left(a_{0},a_{1},a_{2}\right),\quad n\ge2$
from \eqref{eq:2.18}, then we apply $\Gamma(x)\Gamma(1-x)=\frac{\pi}{\sin x}$
and $\Gamma(x+1)=x\Gamma(x)$ to \eqref{eq:2.18} to obtain 
\[
a_{0}=2\pi,a_{2}=\frac{2\pi^{2}}{9\sqrt{3}\Gamma^{2}(\frac{1}{3})},\ a_{1}=\frac{4\cdot2^{\frac{2}{3}}\pi^{\frac{5}{2}}}{\sqrt{3}\Gamma(\frac{1}{6})\Gamma^{2}(\frac{1}{3})}.
\]
By $\Gamma(z)\Gamma(z+1/2)=2^{1-2z}\Gamma(1/2)\Gamma(2z)$ we get
$\Gamma\left(\frac{1}{6}\right)=\frac{2^{\frac{2}{3}}\sqrt{\pi}\Gamma(\frac{1}{3})}{\Gamma(\frac{2}{3})}=\frac{2^{\frac{5}{3}}\Gamma^{2}(\frac{1}{3})}{\sqrt{3}\cdot\sqrt{\pi}},$
which gives $a_{1}=\frac{2\pi^{3}}{\Gamma^{4}(\frac{1}{3})}$. Now
it is easy to see that $a_{n}\in\mathbb{Q}\left(\sqrt{3},\pi,\Gamma\left(\frac{1}{3}\right)\right),\quad n\ge0$,
and example 3 follows from $\lambda_{k}=\frac{1}{i_{k}}$ and \eqref{eq:2.16}.
Example 4 follows from \eqref{eq:2.19} and $\lambda_{k}=\frac{1}{j_{\nu,k}(q)}$,
example 5 follows from \eqref{eq:2.21} and $\lambda_{k}=\frac{1}{i_{k}(q)}$,
example 6 follows from \eqref{eq:2.3} and $\lambda_{k}=\frac{1}{z_{k}}$,
example 7 follows from \eqref{eq:2.9} and $\lambda_{k}=\frac{1}{z_{k}(\chi)}$.

\subsection{Proof of Theorem \ref{thm:3.2}}

By considering $\frac{\lambda_{n}}{\lambda}$ we may assume that $\lambda=1$
and $\sup\left\{ \left|\lambda_{n}\right|:\ n\in\mathbb{N}\right\} \le1$.
Let $\left\{ \lambda_{n}\right\} _{n=1}^{\infty}$ be a sequence of
positive numbers such that ${\displaystyle 0<\sum_{n=1}^{\infty}}\lambda_{n}<\infty$,
we define 
\begin{equation}
\mu(x)=\sum_{n=1}^{\infty}\lambda_{n}\delta\left(x-\lambda_{n}\right),\label{eq:4.14}
\end{equation}
then its moments are given by 
\begin{equation}
m_{k}=\int_{0}^{1}x^{k}d\mu(x)=\sum_{n=1}^{\infty}\lambda_{n}^{k+1}=p_{k+1},\quad k\in\mathbb{N}_{0}.\label{eq:4.15}
\end{equation}
Thus,
\[
\ensuremath{(-1)^{j}\Delta^{j}m_{k}=\int_{0}^{1}x^{k}(1-x)^{j}d\mu(x)\ge0,}
\]
and the assertion \eqref{eq:3.3} is proved. 

Assume \eqref{eq:3.3} holds, then the Hausdorff moment problem \eqref{eq:4.15}
is solvable\cite{Akhiezer,Shohat}, that is, there is a positive measure
$\mu(x)$ on the interval $[0,1]$ satisfying \eqref{eq:4.15}. Then
for $k\in\mathbb{N}_{0}$ and $f(x)={\displaystyle \sum_{j=0}^{k}}c_{j}x^{j},\quad c_{0},c_{1},\dots,c_{k}\in\mathbb{R}$
we have
\[
\int_{0}^{1}f(x)^{2}d\mu(x)=\sum_{i,j=0}^{k}m_{i+j}c_{i}c_{j}=\sum_{i,j=0}^{k}p_{i+j+1}c_{i}c_{j}\ge0.
\]
Since
\[
\sum_{i,j=0}^{k}p_{i+j+1}c_{i}c_{j}=\sum_{n=1}^{\infty}\lambda_{n}\left(f\left(\lambda_{n}\right)\right)^{2},
\]
then,
\[
\int_{0}^{1}f(x)^{2}d\mu(x)=\sum_{n=1}^{\infty}\lambda_{n}\left(f\left(\lambda_{n}\right)\right)^{2}.
\]
If $\left\{ \lambda_{n}\right\} _{n=1}^{\infty}$ is not a sequence
of positive numbers, then at least one of the $\lambda_{n}$s are
negative or with nonzero imaginary parts. Let us eliminate the complex
case first. Assume there are at least one $\lambda_{n}$ with nonzero
imaginary parts, Since ${\displaystyle \sum_{n=1}^{\infty}}\left|\lambda_{n}\right|<\infty$,
then for any $\epsilon>0$ there exist finitely many of them having
imaginary parts bigger or equal to $\epsilon$ in absolute value.
Thus, there exist only finitely many of $\lambda_{n}$s with largest
imaginary parts in absolute value, and they are on one or two horizontal
lines. We list them and their complex conjugates into distinct pairs
as $\left\{ \gamma_{1},\overline{\gamma_{1}}\right\} ,\left\{ \gamma_{2},\overline{\gamma_{2}}\right\} ,\dots,\left\{ \gamma_{n},\overline{\gamma_{N}}\right\} $.
Let
\[
a=\inf\left\{ \left|\rho_{j}-\rho_{k}\right|\bigg|\rho_{j}\neq\rho_{k},\ \rho_{j},\rho_{k}\in\cup_{j=1}^{N}\left\{ \gamma_{j},\overline{\gamma_{j}}\right\} ,\ \Im\rho_{j},\Im\rho_{k}>0\right\} >0,
\]
\[
b=\sup\left\{ \left|\Im z\right|\bigg|z\in\left\{ \lambda_{n}\right\} _{n=1}^{\infty}\backslash\left(\cup_{j=1}^{N}\left\{ \gamma_{j},\overline{\gamma_{j}}\right\} \right)\right\} \ge0,
\]
and 
\[
9d=\min\left\{ a,\left|\Im\gamma_{1}\right|-b\right\} >0.
\]
For any $0<\delta<d$, let 
\[
K_{\delta}=R_{d}\cup_{j=1}^{N}\left(D_{\delta}(\gamma_{j})\cup D_{\delta}(\overline{\gamma_{j})}\right),
\]
where $R_{d}$ is the closed rectangle passing through $\left(-1-d,0\right),\ \left(1+d,0\right),\ \left(0,(b+d)i\right),\ \left(0,-(b+d)i\right)$,
and $D_{\delta}(\gamma_{j})$ and $D_{\delta}(\overline{\gamma_{j}})$
are the closed disks with radius $\delta$ centered at $\gamma_{j}$
and $\overline{\gamma_{j}}$ respectively. Clearly, $K_{\delta}$
is compact, symmetric with respect to real axis, it contains $\left\{ \lambda_{n}\right\} _{n=1}^{\infty}\cup[0,1]$.
It is also clear that $\mathbb{C\backslash}K_{\delta}$ is connected.
Then the function 
\[
f_{\delta}(z)=\frac{1}{z-\gamma_{1}+2\delta}+\frac{1}{z-\overline{\gamma_{1}}+2\delta}
\]
is continuous on $K_{\delta}$ and holomorphic in the interior of
$K_{\delta}$. By the Mergelyan's theorem, \cite{Carleson,Vitushkin},
it can be approximated uniformly on $K_{\delta}$ with polynomials
$p_{M}^{(\delta)}(z)$. Since $K_{\delta}$ is invariant under complex
conjugation, and the function satisfies $f_{\delta}(z)=\overline{f_{\delta}\left(\overline{z}\right)}$,
we may take $p_{M}^{(\delta)}(z)$ with only real coefficients, that
is, $p_{M}^{(\delta)}(z)=\overline{p_{M}^{(\delta)}\left(\overline{z}\right)}$.
Then we have
\begin{equation}
\int_{0}^{1}p_{M}^{(\delta)}(x)^{2}d\mu(x)=\sum_{n=1}^{\infty}\lambda_{n}\left(p_{M}^{(\delta)}(\lambda_{n})\right)^{2}.\label{eq:4.16}
\end{equation}
Let $M\to\infty$ in \eqref{eq:4.16} to get
\[
\int_{0}^{1}f_{\delta}(x)^{2}d\mu(x)=\sum_{n=1}^{\infty}\lambda_{n}\left(f_{\delta}\left(\lambda_{n}\right)\right)^{2}.
\]
Now we observe that
\[
\Re\left\{ \gamma_{1}\left(f_{\delta}\left(\gamma_{1}\right)\right)^{2}\right\} \asymp\begin{cases}
\frac{1}{2\delta},\  & \Re(\gamma_{1})=0,\\
\frac{\Re(\gamma_{1})}{4\delta^{2}},\  & \Re(\gamma_{1})\neq0,
\end{cases}
\]
and
\[
\left(f_{\delta}\left(\lambda_{n}\right)\right)^{2}=\mathcal{O}\left(1\right),\quad\lambda_{n}\neq\gamma_{1}
\]
as $\delta\downarrow0$. Let $m^{+}$ and $m^{-}$ be the multiplicities
of $\gamma_{1}$ and $\overline{\gamma_{1}}$ in the sequence $\left\{ \lambda_{n}\right\} _{n=1}^{\infty}$
, then 
\[
\left(m^{+}+m^{-}\right)\Re\left\{ \gamma_{1}\left(f_{\delta}\left(\gamma_{1}\right)\right)^{2}\right\} =\Re\left\{ \int_{0}^{1}f_{\delta}(x)^{2}d\mu(x)-\sum_{\lambda_{n}\neq\gamma_{1},\overline{\gamma_{1}}}^{\infty}\lambda_{n}\left(f_{\delta}\left(\lambda_{n}\right)\right)^{2}\right\} .
\]
Hence, 
\[
\left|\Re\left\{ \gamma_{1}\left(f_{\delta}\left(\gamma_{1}\right)\right)^{2}\right\} \right|=\mathcal{O}\left(\int_{0}^{1}d\mu(x)+\sum_{\lambda_{n}\neq\gamma_{1},\overline{\gamma_{1}}}^{\infty}\left|\lambda_{n}\right|\right)=\mathcal{O}\left(1\right)
\]
as $\delta\downarrow0$, which is clearly impossible according to
the asymptotic behavior of the left hand side.

Having eliminated the complex case, we now prove that $\lambda_{n}$s
can not be negative either. Assume that there exist at least one negative$\lambda_{n}$,
since ${\displaystyle 0<\sum_{n=1}^{\infty}}\left|\lambda_{n}\right|<\infty$,
we must have one of them, say $\lambda_{k_{0}}$ of multiplicity $m_{0}\ge1$
with largest modulus. Let

\[
a=-\inf\left\{ \Re(z)\bigg|z\in\left\{ \lambda_{n}\right\} _{n=1}^{\infty}\backslash\left\{ \lambda_{k_{0}}\right\} \right\} \ge0\quad9d=\left|\lambda_{k_{0}}\right|-a>0.
\]
For $0<\delta<d$ we consider the continuous function $\frac{1}{z-\lambda_{k_{0}}+2\delta}$
and $K_{\delta}=R_{d}\cup D_{\delta}(\lambda_{k_{0}})$, where $R_{d}$
is the closed rectangle passing through $\left(-a-d,0\right),\ \left(1+d,0\right),\ \left(0,di\right),\ \left(0,-di\right)$,
and $D_{\delta}(\lambda_{k_{0}})$ is the closed disk with radius
$\delta$ centered at $\lambda_{k_{0}}$. The rest of the proof is
similar to the complex case.

\subsection{Proof of Theorem \ref{thm:3.5}}

Since $f(z)$ is an entire function of genus $0$ such that $f(0)=1$,
then $f(z)=\prod_{n=1}^{\infty}\left(1-\frac{z}{z_{n}}\right)$ where
$\left\{ z_{n}\right\} _{n=1}^{\infty}$ are the roots of $f(z)$.
Let $\lambda_{n}=\frac{1}{z_{n}}$ in Theorem \ref{thm:3.2}, then
the sequence $\left\{ z_{n}\right\} _{n=1}^{\infty}$ is a sequence
of positive number if and only if for some $\rho$ such that $0<\rho\le\inf\left\{ \left|z_{n}\right|\bigg|\ n\in\mathbb{N}\right\} $
we have 
\[
\ensuremath{(-1)^{j}\Delta^{j}\left(\rho^{k+1}p_{k+1}\right)\ge0},\quad j,k\in\mathbb{N}_{0},
\]
 where 
\[
p_{k}={\displaystyle \sum_{n=1}^{\infty}}\frac{1}{z_{n}^{k}},\quad k\in\mathbb{N},\quad\Delta m_{n}=m_{n+1}-m_{n}.
\]
 From 
\[
-\frac{\rho f'(\rho z)}{f(\rho z)}=\sum_{k=0}^{\infty}z^{k}\rho^{k+1}p_{k+1}
\]
 and \eqref{eq:4.12} we get
\[
(-1)^{j}\Delta^{j}\left(\rho^{k+1}p_{k+1}\right)=\frac{-\rho}{\left(k+j\right)!}\frac{\partial^{j+k}}{\partial z^{j+k}}\left\{ \left(z-1\right)^{j}\frac{f'(\rho z)}{f(\rho z)}\right\} \bigg|_{z=0}\ge0
\]
for any nonnegative integers $j,\ k$.

\subsection{Proof of Corollary \ref{cor:3.8}}

Since $G(z)$ is an even entire function of genus $0$ or $1$ which
satisfies $G(0)\neq0$ and takes real values for $z\in\mathbb{R}$.
Then for any fixe $c\in\mathbb{R}$, by applying the Lemma stated
in \cite{Gasper-2}, the entire function $G(z-ic)+G(z+ic)$ in variable
$z$ is also of the same type. Clearly, $G(ic)\neq0,\quad c\in\mathbb{R}$,
thus the function 
\[
f(z)=\frac{G(\sqrt{z}-ic)+G(\sqrt{z}+ic)}{2G(ic)}
\]
is an entire function of genus $0$ has only positive zeros. We apply
Theorem \ref{thm:3.5} to this $f(z)$ to get
\[
\frac{\partial^{j+k}}{\partial z^{j+k}}\left\{ \left(z-1\right)^{j}\frac{G'(\rho\sqrt{z}-ic)+G'(\rho\sqrt{z}+ic)}{\sqrt{z}\left(G(\rho\sqrt{z}-ic)+G(\rho\sqrt{z}+ic)\right)}\right\} \bigg|_{z=0}\le0,
\]
 where $c\in\mathbb{R}$ and $0<\rho<\inf\left\{ \left|z\right|\bigg|\ G(z)=0\right\} $.
\begin{acknowledgement*}
This work is partially supported by Chinese National Natural Science
Foundation grant No.11371294.\end{acknowledgement*}


\begin{thebibliography}{10}
\bibitem{Gasper-2} G. Gasper, Using integrals of squares to prove
that certain real-valued special functions to prove that the P\'{o}lya
$\Xi^{*}\left(z\right)$ function, the functions $K_{iz}\left(a\right),\ a>0$
and some other entire functions have only real zeros, dedicated to
Dan Waterman on the occasion of his 80th birthday.
\end{thebibliography}
\end{document}